\documentclass[sn-mathphys]{sn-jnl}

\jyear{2021}%

\theoremstyle{thmstyleone}%
\newtheorem{theorem}{Theorem}
\newtheorem{lemma}[theorem]{Lemma}

\theoremstyle{thmstyletwo}%
\newtheorem{remark}{Remark}%

\theoremstyle{thmstylethree}%
\newtheorem{definition}{Definition}%

\begin{document}

\title[$L^2$-supercritical Kirchhoff equations in bounded domains]{Normalized solutions of $L^2$-supercritical Kirchhoff equations in bounded domains}


\author[]{\fnm{Qun} \sur{Wang}}\email{wq19980228@163.com}

\author*[]{\fnm{Xiaojun} \sur{Chang}}\email{changxj100@nenu.edu.cn}

\equalcont{These authors contributed equally to this work.}

\affil[]{\orgdiv{School of Mathematics and Statistics \& Center for Mathematics and Interdisciplinary Sciences}, \orgname{Northeast Normal University},  \city{Changchun}, \postcode{130024},  \country{China}}


\abstract{~In this paper, we investigate the existence of normalized solutions for the following nonlinear Kirchhoff type problem
\begin{equation*}
\begin{cases}
-(a+b\int_{\Omega}\vert\nabla u\vert^2dx)\Delta u+\lambda u=\vert u\vert^{p-2}u & \text{ in }\Omega,\\
u=0 & \text{ on }\partial\Omega
\end{cases}
\end{equation*}
subject to the constraint $\int_{\Omega}\vert u\vert^2dx=c$. Here, $a$ and $b$ are positive constants, $\Omega$ is a smooth bounded domain in $\mathbb{R}^N$ with $1\leq N\leq3$, $c>0$ is a prescribed value, and $\lambda\in \mathbb{R}$ is a Lagrange multiplier.
In the $L^2$-supercritical regime $2+\frac{8}{N}<p<2^*$, we establish the existence of mountain pass-type normalized solutions. Our approach relies on utilizing a parameterized version of the minimax theorem with Morse index information for constraint functionals, and developing a blow-up analysis for the nonlinear Kirchhoff equations. Furthermore, we explore the asymptotic behavior of these solutions as $b\rightarrow0$.}




\maketitle

\section{Introduction and main results}\label{sec1}

Let $\Omega$ be a bounded domain in $\mathbb{R}^N (N=1,2,3)$ with a smooth boundary $\partial\Omega$. We study the following nonlinear Kirchhoff type problem
\begin{equation}\label{eq1}
\begin{cases}
-(a+b\int_{\Omega}\vert\nabla u\vert^2dx)\Delta u+\lambda u=\vert u\vert^{p-2}u & \text{ in }\Omega,\\
u=0 & \text{ on }\partial\Omega
\end{cases}
\end{equation}
under the constraint
\begin{equation}\label{eq2}
\int_{\Omega}\vert u\vert^2dx=c,
\end{equation}
where $a, b$ are positive constants, $c$ is a prescribed value, and $\lambda\in\mathbb{R}$ is a Lagrange multiplier, $2+\frac{8}{N}<p<2^*$(here $2^*=+\infty$ if $N=1,2$, and $2^*:=\frac{2N}{N-2}$ if $N=3$).

The problem described by equation (\ref{eq1}) is connected to the stationary analog of the equation
\[
\rho\frac{\partial^2u}{\partial t^2}-\bigg(\frac{P_0}{h}+\frac{E}{2L}\int_{0}^{L}\bigg\vert\frac{\partial u}{\partial x}\bigg\vert^2\bigg)\frac{\partial^2u}{\partial x^2}=0,
\]
where $u$ represents the displacement, $L$ is the length of the string, $h$ is the area of the cross-section, $E$ denotes the Young modulus of the material, $\rho$ is the mass density, and $P_0$ denotes the initial tension. This model, initially proposed by Kirchhoff in \cite{Kirchhoff-1883}, is recognized as a form of the classical D'Alembert wave equation. It describes the free vibration of elastic strings and characterizes the motions of moderately large amplitude. Beyond its physical applications, Kirchhoff problems also arise in the field of biological systems. In this context, $u$ describes a process that depends on its own average, such as population density. For further insights into the physical and mathematical background of Kirchhoff-type equations, the readers are referred to \cite{Arosio-1996,Chipot-1997,D’Ancona-1992,Pohozaev-1975} and the references therein.

Problem (\ref{eq1}) is often referred to as nonlocal due to the presence of the term $\int_{\Omega}\vert\nabla u\vert^2dx$. This implies that (\ref{eq1}) is no longer a pointwise identity, leading to mathematical difficulties and providing an especially interesting subject for study. In the last two decades, there has been intensive research on the existence and concentration phenomena for positive and nodal solutions of nonlinear Kirchhoff type problems, as seen in \cite{Deng-2015,Figueiredo-2014,He-2012,Li-2020,Liu-2022,PZ-2006,Zhang-2020}.


In this paper, we focus on the existence of normalized solutions for (\ref{eq1}), i.e., solutions of (\ref{eq1}) satisfying the constraint condition (\ref{eq2}).
To find solutions of (\ref{eq1}) with a prescribed $L^2$-norm, we seek critical points of the energy functional
\begin{equation}\label{eq3}
J(u)=\frac{a}{2}\int_{\Omega}\vert\nabla u\vert^2dx+\frac{b}{4}\bigg(\int_{\Omega}\vert\nabla u\vert^2dx\bigg)^2-\frac{1}{p}\int_{\Omega}\vert u\vert^pdx
\end{equation}
on the space of $L^2$-constrained functions
\begin{equation}\label{eq4}
\mathcal{S}_c:=\left\{u\in H^{1}_0(\Omega)\bigg\vert \ \int_{\Omega}\vert u\vert^2dx=c\right\}.
\end{equation}
In recent years, there have been many studies on normalized solutions of nonlinear Kirchhoff type equations when $\Omega=\mathbb{R}^N$.
In \cite{Ye-MMAS}, Ye
 established the existence of global minimizers of $J\vert_{\mathcal{S}_c}$ when $2<p<2+\frac{8}{N}$, and mountain pass-type critical points of $J\vert_{\mathcal{S}_c}$ when $2+\frac{8}{N}<p<2^*$. Later on, Luo and Wang \cite{Luo-2017} obtained the multiplicity of mountain pass solutions in the mass supercritical case when $N=3$. Subsequently, Qi and Zou \cite{Qi-2022} provided the exact number and expressions of positive normalized solutions. In \cite{He-2023}, He et al. investigated the existence of ground state normalized solutions for (\ref{eq1}) with general nonlinearities in the $L^2$-supercritical case, and they also studied the asymptotic behavior of these solutions as $c\to0$ and $c\to+\infty$ respectively. In a recent result, Feng, Liu and Zhang \cite{Feng-2023} proved the existence and multiplicity of normalized solutions for the case of combined nonlinearities, i.e., the nonlinear term is in the form of $\vert u\vert^{p-2}u+\vert u\vert^{q-2}u$. Additionally, other interesting results on normalized solutions have been explored and further developed in other contexts, for related works, we refer to \cite{Carri-2022, Cai-2023, Chen-2021, Cui-2022, Hu-2021, Li-2019, Li-Nie-2023, Li-R-2023, He-JGA-2023, Zeng-2017}.

When $\Omega$ is bounded, the existence of normalized solutions of nonlinear Kirchhoff type equations was also studied recently.
In \cite{Zhu-2021}, Zhu, Li and Liang considered the following Kirchhoff type transmission problem
\begin{equation*}
\begin{cases}
-k_1(\int_{\Omega_1}\vert\nabla u\vert^2dx)\Delta u=\lambda u+\vert u\vert^{p-2}u & \text{ in }\Omega_1,\\
-k_2(\int_{\Omega_2}\vert\nabla v\vert^2dx)\Delta v=\lambda v+\vert v\vert^{p-2}v & \text{ in }\Omega_2,\\
\int_{\Omega_1}u^2dx+\int_{\Omega_2}v^2dx=\mu^2,\\
v=0 & \text{ on }\Gamma,\\
u=v & \text{ on }\Sigma,\\
k_1(\int_{\Omega_1}\vert\nabla u\vert^2dx)\frac{\partial u}{\partial \nu}=k_2(\int_{\Omega_2}\vert\nabla v\vert^2dx)\frac{\partial v}{\partial \nu} & \text{ on }\Sigma,
\end{cases}
\end{equation*}
where $\Omega_1$ and $\Omega_2$ are two disjoint subdomains of a bounded domain $\Omega\subset\mathbb{R}^3$ with a smooth boundary $\Gamma=\partial\Omega$ such that $\Omega=\bar{\Omega}_1\cup\Omega_2$, $\Sigma=\partial\Omega_1=\bar{\Omega}_1\cup\bar{\Omega}_2\subset\Omega$ is smooth, $\partial\Omega_2=\Gamma\cup\Sigma$, and $\nu$ is the unit outward normal vector to $\Omega_2$ and inward to $\Omega_1$, and $k_i(s)=1+b_is^{m+1}$ for $s\in\mathbb{R}_+$. Here $b_i>0$ for $i=1,2$ and $m\in[1,3)$. The authors considered both the $L^2$-subcritical and critical situations, i.e., $p\in(2,2+\frac{4m}{3}]$, where they established the existence of ground-state normalized solutions within a certain range of $\mu$. In addition, in the $L^2$-supercritical and Sobolev critical case $(2+\frac{4m}{3}, 6]$, they identified a threshold value $\mu(m)>0$ such that for each $\mu\in(0,\mu(m))$, the problem has a normalized solution.
Furthermore, they examined the special case when $\Omega_1=\emptyset$ and $m=2$, leading to a single Kirchhoff equation in the form:
\begin{equation}\label{Kir}
\begin{cases}
-(1+b\int_{\Omega}\vert\nabla u\vert^2dx)\Delta u=\lambda u+\vert u\vert^{p-2}u & \text{ in }\Omega,\\
u=0 & \text{ in }\partial\Omega,\\
\int_{\Omega}u^2dx=\mu^2.
\end{cases}
\end{equation}
Their arguments revealed that, if $p\in (2, \frac{14}{3})$, problem (\ref{Kir}) admits at least a normalized solution for all $\mu>0$. For the case of $p\in[\frac{14}{3},6]$, there exists some $\mu^*>0$ such that (\ref{Kir}) has a local minimum normalized solution for $\mu\in(0,\mu^*)$.

We note that, the functional $J(u)$ is unbounded from below on the constraint set $\mathcal{S}_c$ when $p\in(2+\frac{8}{N},2^*)$. This implies that a mountain pass geometry structure emerges on $\mathcal{S}_c$. The natural question arising from this analysis is whether the Kirchhoff equation in bounded domains possesses a mountain pass-type normalized solution? The aim of the present paper is to provide an affirmative answer to this problem. In fact, for an even more general case, i.e., the bounded domain $\Omega$ belongs to $\mathbb{R}^N$ with $N=1,2,3$, when $p$ stays in the $L^2$-supercritical regime, we shall establish the existence of mountain pass-type normalized solutions of the nonlinear Kirchhoff equation (\ref{eq1}) when the mass $c$ satisfies $0<c<c^*$ for some $c^*>0$.

To obtain mountain pass-type normalized solutions, the key step is to seek a bounded (PS) sequence derived from the mountain pass geometry.
In most prior studies concerning mountain pass-type normalized solutions in $\mathbb{R}^N$, the associated Pohozaev identity serves as a crucial tool for establishing the boundedness of (PS) sequences, see \cite{Ye-MMAS,Luo-2017,Feng-2023} (see also \cite{J-1997, BS-2017, S2020-1,S2020-2} for the case of nonlinear Schr\"odinger equations). However, due to the presence of boundary terms in the corresponding Pohozaev identity in bounded domains, we cannot utilize the Pohozaev identity in the same manner as in the case of $\mathbb{R}^N$. Furthermore, another obstacle in studying mountain pass-type normalized solutions is that the bounded domains lack invariance under translations and dilations. Consequently, the scaling arguments commonly employed in the study of normalized solutions in the entire space do not apply to bounded domains.  Motivated by \cite{Chang-2022}, we will develop a minimax argument based on the monotonicity technique, together with some Morse index estimates and blow up analysis, to establish the existence of mountain pass-type normalized solutions for nonlinear Kirchhoff type equations in bounded domains.

 Our first main result reads as follows.

\begin{theorem}\label{thm:1}
Let $1\leq N\leq3$, $a,b>0$, $2+\frac{8}{N}<p<2^*$, and $\Omega\subset\mathbb{R}^N$ be a bounded domain with a smooth boundary $\partial\Omega$. Then, there exists $c^*>0$ such that, for any given $0<c\leq c^*$, the Kirchhoff equation (\ref{eq1}) with the mass constraint (\ref{eq2}) has a positive solution which corresponds to a mountain pass-type critical point of $J$ on $\mathcal{S}_c$.
\end{theorem}


To prove Theorem \ref{thm:1}, we draw inspiration from \cite{BCJS-2023,Chang-2022}. Our approach involves utilizing a parameterized version of the minimax theorem, as introduced in \cite{Borthwick-2023}. We will undertake a blow-up analysis for the sequence of solutions, ensuring a consistent and bounded Morse index, specifically for the approximating nonlinear Kirchhoff equations. This process leads to obtaining normalized solutions for the original equation (\ref{eq1}).



To illustrate the idea, let's provide a brief outline of our methodology.
Consider the family of functionals $J_{\rho}:\mathcal{S}_c\rightarrow\mathbb{R}$ defined as follows:
$$J_\rho(u)=\frac{a}{2}\int_{\Omega}\vert\nabla u\vert^2dx+\frac{b}{4}\bigg(\int_{\Omega}\vert\nabla u\vert^2dx\bigg)^2-\frac{\rho}{p}\int_{\Omega}\vert u\vert^pdx,$$
where $\rho\in[\frac{1}{2},1]$. Firstly, by applying the abstract minimax theorem from \cite{Borthwick-2023}, we establish the existence of a mountain pass critical point $u_{\rho}$ of $J_{\rho}(\cdot)$ for almost every $\rho\in[\frac{1}{2},1]$. Subsequently, our goal is to take the limit of some solution sequence $\{u_{n}\}$ along a sequence $\rho_n\rightarrow1^-$ to identify mountain pass normalized solutions to the original problem. To accomplish this, we conduct a comprehensive blow-up analysis for solutions of the approximating nonlinear Kirchhoff equations.  We assume, by contradiction, that the $L^{\infty}$ norm of $\{u_n\}$ and thus the corresponding Lagrange multipliers $\{\lambda_n\}$ is unbounded. Through our analysis, we describe the local and global behavior of $\{u_n\}$ near the local maximum points (see Section 3). Subsequently, we derive a contradiction by utilizing the decay properties of $\{u_n\}$ away from these blow-up points and the unboundedness of $\{\lambda_n\}$. This contradiction, in turn, confirms the boundedness of $\{u_{n}\}$, allowing us to assert that ${u_{n}}$ converges to the mountain pass normalized solution of (\ref{eq1}).

It's worth noting that in \cite{Chang-2022}, a detailed blow-up analysis was conducted for sequences of solutions with bounded Morse index for nonlinear Schr\"odinger equations. However, the employed blow-up function, expressed as
\[
U_n(y)=\lambda_n^{-\frac{1}{p-2}}u_n(\lambda_n^{-\frac{1}{2}}y+P_n)
\]
where $P_n$ represents a local maximum point, proves insufficient for obtaining a priori pointwise estimates for nonlinear Kirchhoff equations. Given the presence of a nonlocal term, new challenges arise when dealing with the limiting equation. Inspired by \cite{He-2023}, we introduce another blow-up function represented as
\[
U_n(y)=\lambda_n^{-\frac{1}{p-2}}u_n(\lambda_n^{-\frac{1}{2}}\sqrt{e_n}y+P_n),
\]
with the term $e_n=\int_\Omega\vert\nabla u_n\vert^2dx$ playing a crucial role in the blow-up analysis of the approximating problems. Notably, the limit equation coincidentally transforms into the classical Schr\"odinger equation. However, the introduction of new blow-up functions brings about new challenges, and to address these issues, we will elucidate the relationship between $e_n$ and $\lambda_n$.

Additionally, while investigating the blow-up behavior of $u_n$ in the case of $N=1$, and aiming to establish the decay of $u_n$ away from the local
maximum points, it is crucial to demonstrate the satisfaction of the limit equation in the real line $\mathbb{R}^1$ with a positive linear term. However, in the scenario where $N=1$, the absence of Liouville theorems, unlike the cases when $2 \leq N \leq 3$, poses a challenge. This challenge will be overcome by employing the moving plane method.

In the subsequent analysis, we investigate the asymptotic behavior of the positive solution $u$ obtained in Theorem \ref{thm:1} as $b$ approaches 0. Let $u$ be denoted as $u_b$, and correspondingly, we use $\lambda_b$ for the Lagrange multiplier and $c_b$ for the energy level. The following result is derived.

\begin{theorem}\label{thm:2}
Let $u_b\subset \mathcal{S}_c$ be the positive normalized solution of (\ref{eq1}) obtained by Theorem \ref{thm:1}. Then $u_b\rightarrow u_0$ in $H^1_0(\Omega)$ as $b\rightarrow0$ up to a subsequence, where $u_0\subset \mathcal{S}_c$ is a positive normalized solution of
\begin{equation}\label{e2}
\begin{cases}
-a\Delta u+\lambda u=\vert u\vert^{p-2}u & \text{ \rm{in} }\Omega,\\
u=0 & \text{ \rm{in} }\partial\Omega.
\end{cases}
\end{equation}
\end{theorem}

\begin{remark}\label{remark}
For the case in $\mathbb{R}^N$, Feng, Liu and Zhang \cite{Feng-2023} investigated the asymptotic behavior of normalized solutions as $b\to0$, in which the key tool is to construct a Pohozaev manifold. However, their arguments are not applicable in our case since the appearance of some boundary terms in the corresponding Pohozaev identity.
\end{remark}

We now display the main strategy in proving Theorem \ref{thm:2}. First and foremost, it is crucial to observe that the energy $c_b$ is nondecreasing with respect to $b\in(0,1]$.
Next, we investigate the asymptotic behavior by adapting the proof strategy used in Theorem \ref{thm:1}.
To perform the blow-up analysis as $b_n\to0$, we assume, to the contrary, that $\{u_{b_n}\}$ is unbounded, resulting in the divergence of $\{e_n\}$ to infinity. This, combined with the fact that $b_n\to0$, leads to two scenarios: either the sequence $\{b_ne_n\}$ is bounded or unbounded. Following our earlier discussions, we employ a blow-up argument or utilize the blow-up function introduced in Section \ref{sec:4} to produce a contradiction.

\indent The remaining part of this paper is organized as follows. In Section \ref{sec:3}, we present the existence of mountain pass-type normalized solutions for the approximating problems. In Section \ref{sec:4}, we shall give the blow-up analysis.  Then we complete in Section \ref{sec:5} the proof of Theorems \ref{thm:1}. Finally, in Section \ref{sec:6} we justify the asymptotic behavior result. Throughout this paper, 
$C$ denotes different positive constants whose exact values are not essential to the exposition of arguments.



\section{Mountain pass solutions for approximating problems}\label{sec:3}
This section is devoted to the following perturbed Kirchhoff type problem
\begin{equation}\label{2.1}
\begin{cases}
-\big(a+b\int_{\Omega}\vert\nabla u\vert^2dx\big)\Delta u+\lambda u=\rho\vert u\vert^{p-2}u& \text{ in }\Omega,\\
u=0 & \text{ on }\partial\Omega, \\
\int_{\Omega}\vert u\vert^2dx=c,
\end{cases}
\end{equation}
where $\rho\in [\frac{1}{2}, 1]$. We shall prove that, for almost $\rho\in [\frac{1}{2}, 1]$, problem (\ref{2.1}) admits a mountain pass-type normalized solution, along with some Morse index information.

In the following, we show that the perturbed functional $J_\rho(u)$ has a mountain pass geometry on $\mathcal{S}_c$ uniformly with respect to $\rho\in[\frac{1}{2},1]$.
\begin{lemma}\label{lem:mimax g}
There exists $c^*>0$ such that for each $c\in(0,c^*]$, there exist $w_1,w_2\in \mathcal{S}_c$ independent of $\rho$, such that
\[
c_{\rho}:=\inf_{\gamma\in\Gamma}\max_{t\in(0,1]}J_{\rho}(\gamma(t))>\max\{J_{\rho}(w_1),J_{\rho}(w_2)\}, \forall\rho\in\bigg[\frac{1}{2},1\bigg],
\]
where $\Gamma=\{\gamma\in C([0,1],\mathcal{S}_c)\vert\gamma(0)=w_1,\gamma(1)=w_2\}.$
\end{lemma}
\begin{proof}
From the Gagliardo-Nirenberg inequality, we obtain for every $u\in \mathcal{S}_c$,
\begin{eqnarray}
J_{\rho}(u)
&=&\frac{a}{2}\int_{\Omega}\vert\nabla u\vert^2dx+\frac{b}{4}\bigg(\int_{\Omega}\vert\nabla u\vert^2dx\bigg)^2-\frac{\rho}{p}\int_{\Omega}\vert u\vert^pdx\nonumber\\
&\ge&J(u)\nonumber\\
&\geq&\frac{a}{2}\int_{\Omega}\vert\nabla u\vert^2dx+\frac{b}{4}\bigg(\int_{\Omega}\vert\nabla u\vert^2dx\bigg)^2\nonumber\\
&&-\frac{1}{p}C_pc^{\frac{2N-p(N-2)}{4}}\bigg(\int_{\Omega}\vert\nabla u\vert^2dx\bigg)^{\frac{N(p-2)}{4}}.\label{1-27-1}
\end{eqnarray}

Denote by $\lambda_1(\Omega)$ the first Dirichlet eigenvalue of $(-\Delta, H_0^1(\Omega))$. For any $\alpha\geq\lambda_1(\Omega)$, we define
\begin{eqnarray*}
B_{c\alpha}=\{u\in \mathcal{S}_c\vert\int_{\Omega}\vert\nabla u\vert^2dx<c\alpha\},~~ \partial B_{c\alpha}=\{u\in \mathcal{S}_c\vert\int_{\Omega}\vert\nabla u\vert^2dx= c\alpha\}.
\end{eqnarray*}
Then, the sets
$B_{c\alpha}$ and $\partial B_{c\alpha}$ are both nonempty. In fact, we define the function
\[
\varphi_\alpha(x)=\sqrt{c}\left(\frac{\alpha}{\lambda_1(\Omega)}\right)^{\frac{N}{4}}\varphi_1\left(\sqrt{\frac{\alpha}{\lambda_1(\Omega)}}x\right),
\]
where $\varphi_1$ is the corresponding eigenfunction of $\lambda_1$ such that $\|\varphi_1\|_{L^2(\Omega)}^2=1$.
Clearly, $\varphi_{\frac{\alpha}{2}}\in B_{c\alpha}$ and $\varphi_{\alpha}\in \partial B_{c\alpha}$.

Set \[
c^*=(8\lambda_1(\Omega))^{\frac{N(p-2)-4}{2(2-p)}}\bigg(\frac{ap}{2C_p}\bigg)^{\frac{2}{p-2}}
\]
and
\[
\alpha_0=\frac{1}{2}\bigg(\frac{ap}{2C_p}\bigg)^{\frac{4}{N(p-2)-4}}c^{\frac{2(2-p)}{N(p-2)-4}}.
\]
Clearly, for any
$0<c\leq c^*$, one can easily see $\alpha_0\geq4\lambda_1(\Omega)$. By $p>2+\frac{8}{N}$, we deduce $\frac{N(p-2)-4}{4}>1$. Then, by (\ref{1-27-1}) we get, for any $\rho\in[\frac{1}{2},1]$,
\begin{eqnarray*}
\inf_{u\in\partial B_{c\alpha_0}}J_{\rho}(u)&\ge& \inf_{u\in\partial B_{c\alpha_0}}J(u)\\ &\ge&\frac{a}{2}c\alpha_0+\frac{b}{4}c^2\alpha_0^2-\frac{1}{p}C_pc^{\frac{2N-p(N-2)}{4}}\bigg(c\alpha_0\bigg)^{\frac{N(p-2)}{4}}\\
&=&\bigg[1-\bigg(\frac{1}{2}\bigg)^{\frac{N(p-2)-4}{4}}\bigg]\frac{a}{2}c\alpha_0+\frac{b}{4}c^2\alpha_0^2\\
&\ge&\frac{a}{4}c\alpha_0+\frac{b}{4}c^2\alpha_0^2:=c\beta>0.
\end{eqnarray*}

Let $w_1=\varphi_{\frac{\alpha_0}{4}}$. Then $w_1\in \partial B_{\frac{c\alpha_0}{4}}$.
Moreover, for any $\rho\in [\frac{1}{2},1]$, we have
\begin{eqnarray}\label{mimaxg2}
J_{\rho}(w_1)
&\leq&\frac{a}{2}\int_{\Omega}\vert\nabla w_1\vert^2dx+\frac{b}{4}\bigg(\int_{\Omega}\vert\nabla w_1\vert^2dx\bigg)^2\nonumber\\
&=&\frac{a}{8}c\alpha_0+\frac{b}{64}c^2\alpha_0^2\nonumber\\
&\le&\frac{1}{2}c\beta.
\end{eqnarray}

Now we borrow the idea in \cite[Lemma 3.1]{NTV2019} to construct $w_2$.
 Let $\phi\in C_0^{\infty}(B_1(0))$ with $\phi>0$ in $B_1(0)$ and $\int_{B_1(0)}\phi^{2}dx=1$, and $x_1\in \Omega$. For $k\in \mathbb{N}$, we define
\begin{eqnarray*}\label{mimaxg}
u_k(x)=c^{\frac{1}{2}}k ^{\frac{N}{2}}\phi(k(x-x_1)),x\in \Omega.
\end{eqnarray*}
Obviously, supp$(u_k)\subset B_{\frac{1}{k}}(x_1)\subset\Omega$ for $k$ sufficiently large, so that $u_k\in \mathcal{S}_c$.
By a straightforward calculation, it follows
\[
\vert\vert u_k\vert \vert_{H^1_0(\Omega)}=\sqrt{c}\bigg(k^2\int_{B_1(0)}\vert\nabla\phi\vert^2dx+1\bigg)^{\frac{1}{2}} \to+\infty
\]
and
\begin{eqnarray*}
J_{\rho}(u_k)&\le &J_{\frac{1}{2}}(u_k)\\
&=&\frac{a}{2}ck^2\int_{B_1(0)}\vert\nabla\phi\vert^2dx+\frac{b}{4}c^2k^4\bigg(\int_{B_1(0)}\vert\nabla\phi\vert^2dx\bigg)^2\\
&-&\frac{1}{2p}c^{\frac{p}{2}}k^{\frac{N(p-2)}{2}}\int_{B_1(0)}\vert\phi\vert^{p}dx\\
&\to&-\infty
\end{eqnarray*}
as $k\to+\infty$ since $2+\frac{8}{N}<p<2^*$. So we can take $k_0>0$ large enough, independent of $\rho$, such that, letting $w_2=u_{k_0}$,
 \begin{equation}\label{mimaxg3}
\int_{\Omega}\vert\nabla w_2\vert^2dx>2c\alpha_0 \ \ \text{and} \ \ J_\rho (w_2)\le J_{\frac{1}{2}}(w_2)<0, \forall \rho\in\bigg[\frac{1}{2}, 1\bigg].
\end{equation}

At this point, letting $\Gamma$ and $c_\rho$ be defined as in the statement of the lemma for our choice of $w_1$ and $w_2$, then $\Gamma\ne \emptyset$ holds since
\[
\gamma_0(t)=\frac{\sqrt{c}}{\vert\vert~(1-t)w_1+tw_2\vert\vert_{L^2(\Omega)}}[(1-t)w_1+tw_2], \ \ \ t\in[0,1]
\]
belongs to $\Gamma$. Combining (\ref{mimaxg2}) with (\ref{mimaxg3}), by continuity, for any $\gamma\in\Gamma$ there exists $t_{\gamma}\in[0,1]$ such that $\gamma(t_{\gamma})\in\partial B_{c\alpha_0}$. Therefore, for any $\rho\in[\frac{1}{2},1]$ and $\gamma\in\Gamma$, we have
$$\max_{t\in[0,1]}J_{\rho}(\gamma(t))\geq J_{\rho}(\gamma(t_{\gamma}))\geq\inf_{u\in\partial B_{c\alpha_0}}J(u)\geq c\beta,$$
and thus $c_{\rho}\geq c\beta$. Using (\ref{mimaxg2}) and (\ref{mimaxg3}) it follows that $\max\{J_{\rho}(w_1),J_{\rho}(w_2)\}\leq\frac{c\beta}{2}$, which completes the proof.
\end{proof}

Fix any $\rho\in[\frac{1}{2},1]$. We have the following compactness result.
\begin{lemma}\label{lem:conv}
Let $\{u_n\}\subset \mathcal{S}_c$ be a bounded Palais-Smale sequence of $J_{\rho}\vert_{\mathcal{S}_c}$. Then there exists $u\in \mathcal{S}_c$ such that, up to subsequence, $u_n\to u$ strongly in $H^1_0(\Omega)$.
\end{lemma}
\begin{proof}
Since the sequence $\{u_n\}\subset \mathcal{S}_c$ is bounded in $H^1_0(\Omega)$, up to subsequence, there exists $u\in \mathcal{S}_c$ such that
\begin{eqnarray*}
&& u_n\rightharpoonup u~~\mbox{in}~~ H_0^{1}(\Omega),\\
&& u_n\rightarrow u~~\mbox{in}~~L^{p}(\Omega),~~p\in[2,2^\ast),\\
&& u_n(x)\rightarrow u(x)~~\mbox{for a.e. }x\in\Omega.
\end{eqnarray*}
Define the linear functional $\Psi_{\rho}(u):H^1_0(\Omega)\rightarrow\mathbb{R}$ by
\[
\Psi_{\rho}(u)v=a\int_{\Omega}\nabla u\nabla vdx+b\int_{\Omega}\vert\nabla u\vert^2dx\int_{\Omega}\nabla u\nabla vdx-\rho\int_{\Omega}\vert u\vert^{p-2}uvdx.
\]
Using Lemma 3 in \cite{Berestycki-1983}, there exists $\lambda_n\in\mathbb{R}$ such that
\begin{equation}\label{conv1}
\Psi_{\rho}(u_n)v+\lambda_n\int_{\Omega}u_nv dx=o_n(1)
\end{equation}
for every $v\in H^1_0(\Omega)$. Here $o_n(1)\rightarrow0$ as $n\rightarrow+\infty$.
Taking $v=u_n$, we get
\[
\lambda_nc=\rho\int_{\Omega}\vert u_n\vert^pdx-a\int_{\Omega}\vert\nabla u_n\vert^2dx-b\bigg(\int_{\Omega}\vert\nabla u_n\vert^2dx\bigg)^2+o_n(1).
\]
The boundedness of $\{u_n\}$ in $H^1_0(\Omega)$ implies the boundedness of $\lambda_n$, and thus there exists $\bar{\lambda}\in \mathbb{R}$, up to subsequence, such that $\lambda_n\rightarrow\bar{\lambda}$. Then, replacing $v$ with $u_n-u$ in (\ref{conv1}), one has
\[
\Psi_{\rho}(u_n)(u_n-u)+\lambda_n\int_{\Omega}u_n(u_n-u)dx=o_n(1).
\]
Now, let us define the operator $T: H^1_0(\Omega)\rightarrow\mathbb{R}$ by
\[
T(v)=a\int_{\Omega}\nabla u\nabla v dx+b\int_{\Omega}\vert\nabla u_n\vert^2dx\int_{\Omega}\nabla u\nabla vdx+\bar{\lambda}\int_{\Omega}uvdx.
\]
Note that $T(u_n-u)=o_n(1)$ since $u_n\rightharpoonup u$ in $H^1_0(\Omega)$,
hence, it follows that
\begin{eqnarray*}
o_n(1)&=&\Psi_{\rho}(u_n)(u_n-u)+\lambda_n\int_{\Omega}u_n(u_n-u)dx-T(u_n-u)\notag\\
&=&a\int_{\Omega}\vert\nabla (u_n-u)\vert^2dx+b\int_{\Omega}\vert\nabla u_n\vert^2dx\int_{\Omega}\vert\nabla (u_n-u)\vert^2dx\notag\\
&-&\rho\int_{\Omega}\vert u_n\vert^{p-2}u_n(u_n-u)dx+\lambda_n\int_{\Omega}u_n(u_n-u)dx-\bar{\lambda}\int_{\Omega}u(u_n-u)dx\notag\\
&=&a\int_{\Omega}\vert\nabla (u_n-u)\vert^2dx+b\int_{\Omega}\vert\nabla u_n\vert^2dx\int_{\Omega}\vert\nabla (u_n-u)\vert^2dx+o_n(1)\notag
\end{eqnarray*}
by the strong convergence in $L^p(\Omega)$ of $\{u_n\}$ and the convergence of $\lambda_n$, which concludes the proof.
\end{proof}

To establish the existence of normalized solutions of the mountain pass-type for problem (\ref{2.1}), we will employ a recently introduced minimax principle for constraint functionals by \cite{Borthwick-2023}. In order to present the abstract theorem, in what follows we outline a general framework as in \cite{Berestycki-1983} (see also \cite{Borthwick-2023,BCJS-2023, Chang-2022}).

Let $(E,\langle\cdot,\cdot\rangle)$ and $(H,\langle\cdot,\cdot\rangle)$ be two infinite-dimensional Hilbert spaces such that
$ E\hookrightarrow H\hookrightarrow E^{'}$ with continuous injections. For simplicity, we suppose that $E\hookrightarrow H$ has norm at most $1$ and identify $E$ with its image in $H$.  We define
$$\left\{\aligned &\vert\vert u\vert\vert^2=\langle u,u\rangle,\\
& \vert u\vert^2=(u,u)\endaligned\right.$$
for each $u\in E$,
and
$$\mathcal{M}_c=\{u\in E\big\vert\vert u\vert^2=c\},~ \forall c>0.$$
Obviously, $\mathcal{M}_c$ is a submanifold of $E$ of codimension $1$ and its tangent space at a given point $u\in \mathcal{M}_c$ is given by
$$T_u\mathcal{M}_c=\{v\in E\big\vert(u,v)=0\}.$$
Denote by $\vert\vert\cdot\vert\vert_\ast$ and $\vert\vert\cdot\vert\vert_{\ast\ast}$, respectively, the operator norm of $\mathcal{L}(E,\mathbb{R} ) $ and $\mathcal{L}(E,\mathcal{L}(E,\mathbb{R} )) $.

\begin{definition}\cite{Borthwick-2023}\label{def:hc}
Let $\phi:E\rightarrow\mathbb{R}$ be a $C^2$-functional on $E$ and $\alpha\in (0,1]$. We say that $\phi'$ and $\phi''$ are $\alpha$-H\"{o}lder continuous on bounded sets if for any $R<0$ one can find $M=M(R)>0$ such that, for any $u_1,u_2\in B_R(0)$,
\begin{equation}\label{hc}
\vert\vert\phi{'}-\phi{'}\vert\vert_{\ast}\leq M\vert\vert u_1-u_2\vert\vert^{\alpha},\vert\vert\phi{''}-\phi{''}\vert\vert_{\ast\ast}\leq M\vert\vert u_1-u_2\vert\vert^{\alpha}.
\end{equation}
\end{definition}

\begin{definition}\cite{Borthwick-2023}\label{def:cbm }
Let $\phi$ be a $C^2$-functional on $E$. For any $u\in E$, we define the continuous bilinear map
\begin{equation*}
D^2\phi(u)=\phi{''}(u)-\frac{\phi{'}(u)\cdot u}{\vert u\vert^2}(\cdot,\cdot).
\end{equation*}
\end{definition}

\begin{definition}\cite{Borthwick-2023}\label{def:a morse }
Let $\phi$ be a $C^2$-functional on $E$. For any $u\in \mathcal{M}_c$ and $\theta>0$, we define the approximate Morse index by
\begin{eqnarray*}\label{a morse}
\tilde{m}_{\theta}(u)&=&\sup\{\text{dim}L\vert L\text{ is a subspace of }T_u\mathcal{M}_c\text{ such that }D^2\phi(u)[\varphi,\varphi]<-\theta\vert\vert\varphi\vert\vert^2,\\
&&\forall\varphi\in L\backslash \{0\}\}.
\end{eqnarray*}
If $u$ is a critical point for the constrained functional $\phi\vert_{\mathcal{M}_c}$ and $\theta=0$, we say that is the Morse index of $u$ as constrained critical point.
\end{definition}

\begin{lemma}\cite{Borthwick-2023}\label{lem:mima lem}
Let $I\subset(0,+\infty)$ be an interval and consider a family of $C^2$ functional $\Phi_{\rho}:E\rightarrow\mathbb{R}$ of the form
\begin{equation*}
\Phi_{\rho}(u)=A(u)-\rho B(u), \ \ \ \rho\in I,
\end{equation*}
where $B(u)\geq0$ for every $u\in E$, and
\begin{equation}\label{ab}
\text{either }A(u)\rightarrow+\infty \ \text{or} \ B(u)\rightarrow+\infty \ \text{as} \ u\in E\text{ and } \vert\vert u\vert\vert\rightarrow+\infty.
\end{equation}
Suppose moreover that $\Phi'_{\rho}$ and $\Phi''_{\rho}$ are $\alpha$-H\"{o}lder continuous on bounded sets for some $\alpha\in(0,1]$. Finally, suppose that there exist $w_1,w_2\in\mathcal{M}_c$ (independent of $\rho$) such that, setting
\[
\Gamma=\{\gamma\in C([0,1],\mathcal{M}_c)\vert\gamma(0)=w_1,\gamma(1)=w_2\},
\]
we have
\[
c_{\rho}=\inf_{\gamma\in \Gamma}\max_{t\in[0,1]}\Phi_{\rho}(\gamma(t))>\max\{\Phi_{\rho}(w_1),\Phi_{\rho}(w_2)\}, \ \ \ \rho\in I.
\]
Then, for almost every $\rho\in I$, there exist sequences $\{u_n\}\subset\mathcal{M}_c$ and $\zeta_n\rightarrow0^+$ such that, as $n\rightarrow+\infty$,
\begin{itemize}
\item[(i)] $\Phi_{\rho}(u_n)\rightarrow c_{\rho}$;
\item[(ii)] $\vert\vert\Phi'_{\rho}\vert_{\mathcal{M}_c}(u_n)\vert\vert_*\rightarrow 0$;
\item[(iii)] $\{u_n\}$ is bounded in $E$;
\item[(iv)] $\tilde{m}_{\zeta_n}(u_n)\leq1$.
\end{itemize}
\end{lemma}

\begin{lemma}\label{lem:morse}
Let $0<c\leq c^*$. For almost every $\rho\in[\frac{1}{2},1]$, there exists $(u_{\rho},\lambda_{\rho})\in \mathcal{S}_c\times\mathbb{R}$ which solves (\ref{2.1}).
Moreover, $m(u_{\rho})\leq2$.
\end{lemma}
\begin{proof}
We apply Lemma \ref{lem:mima lem} to the family of functionals $J_{\rho}$, with $E=H^1_0(\Omega), H=L^2(\Omega), \mathcal{M}_c=\mathcal{S}_c$, and $\Gamma$ defined in Lemma \ref{lem:mimax g}. Setting
\[
A(u)=\frac{a}{2}\int_{\Omega}\vert\nabla u\vert^2dx+\frac{b}{4}\bigg(\int_{\Omega}\vert\nabla u\vert^2dx\bigg)^2 \ \ \ \text{and} \ \ \ B(u)=\frac{1}{p}\vert u\vert^pdx,
\]
it is easily seen that assumption (\ref{ab}) holds. Denote by $J'_{\rho}$ and $J''_{\rho}$ respectively the unconstrained first and second derivatives of $J_{\rho}$. Since $J_{\rho}$ is of class $C^2$ on $\mathcal{S}_c$, it follows that $J'_{\rho}$ and $J''_{\rho}$ are local H\"older continuous on $\mathcal{S}_c$. Thus all the assumptions of Lemma \ref{lem:mima lem} are satisfied.

At this point, for almost $\rho\in[\frac{1}{2},1]$, there exists a bounded Palais-Smale sequence $\{u_{n,\rho}\}\subset H^1_0(\Omega)$ for $J_{\rho}$ constrained on $\mathcal{S}_c$ at the level $c_{\rho}$ and a sequence $\{\zeta_n\}\subset\mathbb{R}^+$ with $\zeta_n\rightarrow0^+$ satisfying $\tilde{m}_{\zeta_n}(u_{n,\rho})\leq1$.
For simplicity, we still denote $u_{n,\rho}$ by $u_n$. Since the map $u\mapsto\vert u\vert$ is continuous, and $J_{\rho}(u)=J_{\rho}(\vert u\vert)$, we may assume that
$u_n\geq0$. From Lemma \ref{lem:conv}, we derive that there exists $0\leq u_{\rho}\in \mathcal{S}_c$ such that $u_n\rightarrow u_{\rho}$ in $H^1_0(\Omega)$, which implies that, for some $\lambda_{\rho}\in \mathbb{R}$, $(u_{\rho}, \lambda_{\rho})$ satisfies
\begin{equation*}
\begin{cases}
-\big(a+b\int_{\Omega}\vert\nabla u_\rho\vert^2dx\big)\Delta u_\rho+\lambda_{\rho}u_\rho=\rho\vert u_\rho\vert^{p-2}u_\rho& \text{ in }\Omega,\\
u_\rho=0 & \text{ on }\partial\Omega, \\
\int_{\Omega}\vert u_\rho\vert^2dx=c.
\end{cases}
\end{equation*}
By the strong maximum principle it follows that $u_{\rho}>0$ in $\Omega$. Furthermore, by similar arguments as in \cite{BCJS-2023} we infer that
$m(u_{\rho})\leq2$.
\end{proof}

\section{Blow-up analysis}\label{sec:4}

Utilizing Lemma \ref{lem:morse}, it is established that there exists a sequence $\rho_n \rightarrow 1^-$ and a corresponding sequence of mountain pass critical points $\{u_{\rho_n}\}$ of $J_{\rho_n}$ on $\mathcal{S}_c$ at the level $c_{\rho_n}$ with a Morse index $m(u_{\rho_n}) \leq 2$. Our subsequent objective is to demonstrate the convergence of $\{u_{\rho_n}\}$ to a critical point $u$ of $J\vert_{\mathcal{S}_c}$. To achieve this, a crucial step involves establishing the boundedness of $\{u_{\rho_n}\}$. In this section, we adapt the arguments presented in \cite{Esposito-2011} (see also \cite{PV-2017}) to conduct a blow-up analysis for nonlinear Kirchhoff type equations. This analysis is essential in our efforts to ensure the boundedness of $\{u_{\rho_n}\}$.

For notational convenience, we put $u_n:=u_{\rho_n}$, $c_n:=c_{\rho_n}$, $J_n:=J_{\rho_n}$ and
\[
e_n:=\int_{\Omega}\vert\nabla u_n\vert^2dx.
\]
In what follows, we consider the behaviour of a sequence of solutions $\{u_n\}\subset \mathcal{S}_c$ of the following problem
\begin{equation}\label{eq5.1}
\begin{cases}
-(a+be_n)\Delta u_n+\lambda_n u_n=\rho_nu_n^{p-1}& \text{ in }\Omega,\\
u_n>0  & \text{ in }\Omega, \\
\ u_n=0  & \text{ on }\partial\Omega,
\end{cases}
\end{equation}
where $\rho_n\rightarrow1^-$, $m(u_n)\leq2$ for all $n\in\mathbb{N}$.

By (\ref{eq5.1}) and $c_n=J_n(u_n)$, we deduce that
\begin{equation}\label{energy}
\big(\frac{1}{2}-\frac{1}{p}\big)ae_n+\big(\frac{1}{4}-\frac{1}{p}\big)be_n^2=\frac{c\lambda_n}{p}+c_{n}.
\end{equation}
Using the monotonicity of $c_{\rho}$ with respect to $\rho$, we obtain that $\{c_n\}$ is bounded. In view of $p>2+\frac{8}{N}>4$, it follows that the left side of (\ref{energy}) is nonnegative, which implies $\lambda_n\ge -C$ for some $C>0$. And then we can derive that  $\{e_n\}$ is bounded if $\{\lambda_n\}$ is bounded from above. Furthermore, by (\ref{energy}) we have

\begin{lemma}\label{lem:rel2}
Let $\lambda_n\rightarrow+\infty$ as $n\rightarrow+\infty$. Then, there exists constants $C_1,C_2\geq0$ such that
\[
C_1\le \frac{e_n^{2}}{\lambda_n}\le C_2
\]
holds for $n$ large enough.
\end{lemma}

We also have the following observation.
\begin{lemma}\label{lem:rel1}
Let $P_n$ be a local maximum point for $u_n$. Then
\[
u_n(P_n)\geq\lambda_n^{\frac{1}{p-2}}.
\]
\end{lemma}
\begin{proof}
The fact that $u_n$ is a $C^2$ function in $\Omega$ and $P_n$ is a local maximum point for $u_n$ imply $\Delta u_n(P_n)\leq0$. By (\ref{eq5.1}) we get
\[
\lambda_n u_n(P_n)-\rho_nu_n^{p-1}(P_n)=(a+be_n)\Delta u_n(P_n)\leq0,
\]
which gives the conclusion.
\end{proof}
\begin{remark}\label{rem:rel}
Let $P_n$ be a local maximum point for $u_n$ and $\lambda_n\rightarrow+\infty$ as $n\rightarrow+\infty$. Then by Lemma \ref{lem:rel2} and Lemma \ref{lem:rel1} we have
\begin{eqnarray}\label{1-29-1}
\frac{\sqrt{e_n}}{(u_n(P_n))^{\frac{p-2}{2}}}\leq\frac{\sqrt{e_n}}{\sqrt{\lambda_n}}\rightarrow0 \ \ \ \text{as }n\rightarrow+\infty.
\end{eqnarray}
\end{remark}

The next lemma describes the asymptotic behaviour of the sequence $\{u_n\}$ close to a local maximum point as $\lambda_n\rightarrow+\infty$.

\begin{lemma}\label{lem:blow1}
Let $\lambda_n\rightarrow+\infty$ as $n\rightarrow+\infty$ and $P_n\in\Omega$ be such that, for some $R_n\rightarrow+\infty$,
\[
u_n(P_n)=\max_{\Omega\cap B_{R_n\tilde{\varepsilon}_n\sqrt{e_n}}(P_n)}u_n, \ \ \ \text{ where }\tilde{\varepsilon}_n=(u_n(P_n))^{-\frac{p-2}{2}}\rightarrow0.
\]
Set
\[
U_n(y)=\varepsilon_n^{\frac{2}{p-2}}u_n(\varepsilon_n\sqrt{e_n}y+P_n)\text{ for }y\in\Omega_n=\frac{\Omega-P_n}{\varepsilon_n\sqrt{e_n}}\text{ with } \varepsilon_n=\lambda_n^{-\frac{1}{2}}.
\]
Then, up to a subsequence, we have
\begin{itemize}
\item[(i)] $\frac{\varepsilon_n\sqrt{e_n}}{dist(P_n,\partial\Omega)}\rightarrow0$ as $n\rightarrow+\infty$.
\item[(ii)] $u_n(P_n)=\max\limits_{\Omega\cap B_{R_n\varepsilon_n\sqrt{e_n}}(P_n)}u_n$ for some $R_n\rightarrow+\infty$.
\item[(iii)] $U_n\rightarrow U$ in $C^1_{\rm{loc}}(\mathbb{R}^N)$ as $n\rightarrow+\infty$, where $U$ solves
\begin{equation}\label{eq5.2}
\begin{cases}
-b\Delta U+U=U^{p-1}  & \text{ \rm{in} }\mathbb{R}^N,\\
0<U\leq U(0)  & \text{ \rm{in} }\mathbb{R}^N, \\
U\rightarrow0  & \text{ \rm{as} }\vert x\vert\rightarrow+\infty.
\end{cases}
\end{equation}
\item[(iv)] There exists $\phi_n\in C^\infty_0(\Omega)$ with supp$\phi_n\subset B_{R\varepsilon_n\sqrt{e_n}}(P_n)$, $R>0$, such that for $n$ large
\begin{equation}\label{blow1}
(a+be_n)\int_\Omega\vert\nabla\phi_n\vert^2dx+\int_\Omega[(\lambda-(p-1)\rho_nu_n^{p-2})\phi_n^2]dx<0.
\end{equation}
\item[(v)] For all $R>0$ and $q\geq1$ there holds
\begin{equation}\label{blow2}
\lim_{n\rightarrow+\infty}\lambda_n^{\frac{N}{2}-\frac{q}{p-2}}e_n^{-\frac{N}{2}}\int_{B_{R\varepsilon_n\sqrt{e_n}}(P_n)}u_n^qdx=\int_{B_{R}(0)}U^qdx.
\end{equation}
\end{itemize}
\end{lemma}

\begin{proof}
Let $\tilde{U}_n$ be defined by
\[
\tilde{U}_n(y):=\tilde{\varepsilon}_n^\frac{2}{p-2}u_n(\tilde{\varepsilon}_n\sqrt{e_n}y+P_n) \ \ \ \text{for } y\in\tilde{\Omega}_n:=\frac{\Omega-P_n}{\tilde{\varepsilon}_n\sqrt{e_n}},
\]
and let $d_n$ denote $dist(P_n,\partial\Omega)$. In view of (\ref{1-29-1}), $\tilde{\varepsilon}_n\sqrt{e_n}\to 0$ as $n\to+\infty$.
Assume that
\[
\frac{\tilde{\varepsilon}_n\sqrt{e_n}}{d_n}\rightarrow L\in[0,+\infty].
 \]
Then $\tilde{\Omega}_n\rightarrow H$, where $H=\mathbb{R}^N$ if $L=0$, and $H$ represents a half-space with $0\in \overline{H}$ and $dist(0,\partial H)=\frac{1}{L}$ if $L>0$. Notice that $\tilde{U}_n$ satisfies
\begin{equation*}
\begin{cases}
-(ae^{-1}_n+b)\Delta\tilde{U}_n+\lambda_n\tilde{\varepsilon}_n^2\tilde{U}_n=\rho_n\tilde{U}_n^{p-1}  & \text{ in }\tilde{\Omega}_n,\\
0<\tilde{U}_n\leq \tilde{U}_n(0)=1  & \text{ in }\tilde{\Omega}_n\cap B_{R_n}(0), \\
\tilde{U}_n=0  & \text{ on }\partial\tilde{\Omega}_n,
\end{cases}
\end{equation*}
we derive
\[
0\leq-(ae^{-1}_n+b)\Delta\tilde{U}_n(0)=\rho_n-\lambda_n\tilde{\varepsilon}_n^2,
\]
which implies that $\lambda_n\tilde{\varepsilon}_n^2\in[0,1]$. Then, there exists $\tilde{\lambda}\in[0,1]$ such that up to subsequence,
\[
\lambda_n\tilde{\varepsilon}_n^2\rightarrow\tilde{\lambda} \ \ \ \text{as }n\rightarrow+\infty.
\]
So by elliptic regularity theory \cite{Gilbarg-1977}, we obtain that, up to subsequence, $\tilde{U}_n\rightarrow\tilde{U}$ in $C^1_{\text{loc}}(\bar{H})$ as $n\rightarrow+\infty$, where $\tilde{U}$ solves
\begin{equation}\label{N=1}
\begin{cases}
-b\Delta\tilde{U}+\tilde{\lambda}\tilde{U}=\tilde{U}^{p-1}  & \text{ in }H,\\
0<\tilde{U}\leq \tilde{U}(0)=1  & \text{ in }H, \\
\tilde{U}=0  & \text{ on }\partial H.
\end{cases}
\end{equation}

Now, we claim $m(\tilde{U})\leq2$. In fact, assume that $\phi_1,\phi_2\in C^\infty_0(H)$ are orthogonal in $L^2(H)$ such that
\[
b\int_{H}\vert\nabla\phi_i\vert^2dx+\int_{H}[\tilde{\lambda}-(p-1)\tilde{U}^{p-2}]\phi_i^2dx<0,~ \ \forall i=1,2.
\]
Define
\[
\phi_{i,n}(x):=\tilde{\varepsilon}_n^{-\frac{N-2}{2}}e_n^{-\frac{N}{4}}\phi_i\bigg(\frac{x-P_n}{\tilde{\varepsilon}_n\sqrt{e_n}}\bigg).
\]
Then it is easily seen that $\phi_{1,n}(x)$ and $\phi_{2,n}(x)$ are orthogonal in $L^2(\Omega)$, and
\begin{eqnarray*}
&&(a+be_n)\int_{\Omega}\vert\nabla\phi_{i,n}\vert^2dx+\int_{\Omega}[\lambda_n-(p-1)\rho_nu_n^{p-2}]\phi_{i,n}^2dx\notag\\
&=&(ae_n^{-1}+b)\int_{\tilde{\Omega}_n}\vert\nabla\phi_i\vert^2dx+\int_{\tilde{\Omega}_n}[\lambda_n\tilde{\varepsilon}_n^2-(p-1)\rho_n\tilde{U}_n^{p-2}]\phi_i^2dx\notag\\
&\rightarrow& b\int_{H}\vert\nabla\phi_i\vert^2dx+\int_{H}[\tilde{\lambda}-(p-1)\tilde{U}^{p-2}]\phi_i^2dx\notag\\
&<&0\notag
\end{eqnarray*}
as $n\rightarrow+\infty$ for all $i=1,2$. Thus, $m(\tilde{U})\leq m(u_n)\leq2$.

In the following, we shall prove that $H=\mathbb{R}^N$ and $\tilde{\lambda}>0$.

For the case $N=1$, we first prove $H=\mathbb{R}$. If not, we have $H=[-\frac{1}{L},+\infty)$. Since $\tilde{U}$ is of finite Morse index, using standard regularity theory, we deduce $\tilde{U}\to0$ as $\vert x\vert\to+\infty$. To produce a contradiction, we shall prove that $\tilde{U}$ is monotonically increasing with respect $x$ by the moving plane method. In fact, in this sense, $\tilde{\lambda}\geq0$ and $\tilde{U}$ satisfies
\begin{equation*}
\begin{cases}
-b\tilde{U}''+\tilde{\lambda}\tilde{U}=\tilde{U}^{p-1}  & \text{ in }[-\frac{1}{L},+\infty),\\
0<\tilde{U}\leq \tilde{U}(0)=1  & \text{ in }[-\frac{1}{L},+\infty), \\
\tilde{U}(-\frac{1}{L})=0.
\end{cases}
\end{equation*}
For $\sigma>-\frac{1}{L}$, we set,
\[
\Sigma_{\sigma}=\{x\ge -\frac{1}{L}\vert x<\sigma\} \ \ \ \text{and} \ \ \ x^{\sigma}=2\sigma-x.
\]
Define
\[
\tilde{U}_{\sigma}(x)=\tilde{U}(x^{\sigma}) \ \ \ \text{and} \ \ \ W_{\sigma}(x)=\tilde{U}_{\sigma}(x)-\tilde{U}(x).
\]
Notice that $\tilde{U}_{\sigma}$ satisfies the same equation as $\tilde{U}$, we can obtain that $W_{\sigma}$ satisfies
\[
-W_{\sigma}''+C(x,\sigma)W_{\sigma}=0, x\in \Sigma_{\sigma},
\]
where
\[
C(x,\sigma)=\frac{1}{b}\left(\tilde{\lambda}-\frac{\tilde{U}^{p-1}_{\sigma}(x)-\tilde{U}^{p-1}(x)}{\tilde{U}_{\sigma}(x)-\tilde{U}(x)}\right).
\]
Since $\tilde{U}$ is bounded in $\Sigma_{\sigma}$, there exists constant $C$ such that $\vert C(x,\sigma)\vert\leq C$.

We start the moving planes from $x=-\frac{1}{L}$.  Take $\sigma>-\frac{1}{L}$ to be close to $-\frac{1}{L}$ such that $C(x,\sigma)>-\frac{1}{l^2}$, where $l:=meas\left(\Sigma_{\sigma}\right)=\sigma+\frac{1}{L}$. Since $W_{\sigma}(x)\geq0$ on $\partial\Sigma_{\sigma}$, we can derive that $W_{\sigma}(x)\geq0$ for any $x\in\Sigma_{\sigma}$. In fact, for any $x\in\Sigma_{\sigma}$, we take $v(x)=\frac{W_{\sigma}(x)}{\psi(x)}$. Here
$\psi(x)=\sin\frac{x+\frac{1}{L}}{l}>0$ satisfies $ -\psi''=\frac{1}{l^2}\psi$.
Then, we can find a minimum point $x_0\in\Sigma_{\sigma}$ of $v$ such that $v(x_0)<0$ if the claim is not true. By direct calculation, it follows that
\[
-v''=2v'\frac{\psi'}{\psi}+\frac{1}{\psi}\left(-W''_{\sigma}+\frac{\psi''}{\psi}W_{\sigma}\right).
\]
Since $-v''(x_0)\leq0$ and $v'(x_0)=0$, we infer $-W''_{\sigma}(x_0)+\frac{\psi''}{\psi}W_{\sigma}(x_0)\le 0$. On the other hand, using $W_{\sigma}(x_0)<0$, we have
\begin{eqnarray*}
-W''_{\sigma}(x_0)+\frac{\psi''}{\psi}W_{\sigma}(x_0)&=&-W''_{\sigma}(x_0)-\frac{1}{l^2}W_{\sigma}(x_0)\\
&>&-W''_{\sigma}(x_0)+C(x_0,\sigma)W_{\sigma}(x_0)\\
&=&0,
\end{eqnarray*}
which gives a contradiction.

Next we continue the moving planes. Let us set
\[
\bar{\sigma}=\sup\{\sigma\vert W_{\sigma}(x)\geq0,\forall x\in \Sigma_{\sigma}\}.
\]
To prove $\tilde{U}$ is monotonically increasing with respect $x$, we have to prove that $\bar{\sigma}=+\infty$. Suppose by contradiction that $\bar{\sigma}<+\infty$. Then, we obtain that $W_{\bar{\sigma}}(x)=0$ for all $x\in \Sigma_{\bar{\sigma}}$. If not, there exists $\sigma>\bar{\sigma}$ such that $W_{\sigma}(x)\geq0$ holds which contradicts with the definition of $\bar{\sigma}$. Hence we get that
\[
\tilde{U}(2\bar{\sigma}+\frac{1}{L})=\tilde{U}(-\frac{1}{L})=0,
\]
this contradicts $\tilde{U}>0$ in $(-\frac{1}{L},+\infty)$. This implies that $H=\mathbb{R}$.

Furthermore, we will prove $\tilde{\lambda}>0$. By phase plane analysis, the equation $-b\tilde{U}^{''}+\tilde{U}^{p-1}=0$ has only periodic non-trivial sign-changing solution, so $\tilde{\lambda}=0$ is impossible.

For the case $2\leq N\leq3$, we point out that $\tilde{\lambda}>0$ holds no matter whether $H$ is half-space or whole space. Indeed, we only need to use \cite[Theorem 3]{Bahri-1992} to rule out the fact $\tilde{\lambda}=0$ since $\tilde{U}$ is not trivial. Then, from \cite[Theorem 1.1]{Esposito-2011}, we can derive $H=\mathbb{R}^N$.

Then, we conclude
$$\left(\frac{\tilde{\varepsilon}_n}{\varepsilon_n}\right)^2=\lambda_n\tilde{\varepsilon}_n^2\rightarrow\tilde{\lambda}\in(0,1] \ \ \ \text{as }n\rightarrow+\infty.$$
Since $U_n$ is a solution of
\begin{equation*}
\begin{cases}
-(ae^{-1}_n+b)\Delta U_n+U_n=\rho_n{U}_n^{p-1}  & \text{ in }{\Omega}_n,\\
0<{U}_n\leq {U}_n(0)=\left(\frac{\tilde{\varepsilon}_n}{\varepsilon_n}\right)^{-\frac{2}{p-2}}  & \text{ in }\Omega_n\cap B_{R_n\frac{\tilde{\varepsilon}_n}{\varepsilon_n}}(0), \\
{U}_n=0  & \text{ on }\partial{\Omega}_n,
\end{cases}
\end{equation*}
together with regularity theory and Sobolev embedding, up to subsequence if necessary, one has ${U}_n\rightarrow U$ in $C^1_{\text{loc}}(\bar{H})$ as $n\rightarrow+\infty$, where $U$ solves
\begin{equation*}
\begin{cases}
-b\Delta U+U=U^{p-1}  & \text{ in }H,\\
0<U\leq U(0)  & \text{ in }H, \\
U=0  & \text{ on }\partial H.
\end{cases}
\end{equation*}
Thus, processing in an analogous manner as before, one can see that
\[
\frac{\varepsilon_n\sqrt{e_n}}{d_n}\rightarrow0\text{ as }n\rightarrow+\infty,
\]
which gives $H=\mathbb{R}^N$.
Moreover, we infer $m({U})\leq\sup\limits_{n}m(u_n)\leq2$ and $U\to0$ as $\vert x\vert\to+\infty$. As a consequence, there exists $\phi\in C^\infty_0(\mathbb{R}^N)$ with supp$\phi\subset B_R(0)$ for some $R>0,$ such that
\[
b\int_{\mathbb{R}^N}\vert\nabla\phi\vert^2dx+\int_{\mathbb{R}^N}[1-(p-1)U^{p-2}]\phi^2dx<0.
\]
Considering the function
\[
\phi_{n}(x):={\varepsilon}_n^{-\frac{N-2}{2}}e_n^{-\frac{N}{4}}\phi\bigg(\frac{x-P_n}{{\varepsilon}_n\sqrt{e_n}}\bigg),
\]
one can easily see that $\phi_{n}(x)$ satisfies (\ref{blow1}) for $n$ large enough.

Finally, it suffices to prove (\ref{blow2}). By a change of variables in the integrals and Sobolev compact embeddings, we obtain
\[
\int_{B_{R}(0)}U^qdx=\lim_{n\rightarrow+\infty}\int_{B_{R}(0)}U_n^qdx=\lim_{n\rightarrow+\infty}\lambda_n^{\frac{N}{2}-\frac{q}{p-2}}e_n^{-\frac{N}{2}}\int_{B_{R\varepsilon_n\sqrt{e_n}}(P_n)}u_n^qdx
\]
for all $R>0$ and $q\geq1$.
\end{proof}

In what follows, we establish a global behaviour of $\{u_n\}$ and show that $\{u_n\}$ decays exponentially away from the local maximum points.

\begin{lemma}\label{lem:blow2}
Assume that $\lambda_n\rightarrow+\infty$ as $n\to+\infty$. Then there exist $k\in\{1,2\}$, and sequences of points $\{P^i_n\}$, $i=\{1,k\}$, such that
\begin{equation}\label{blow4}
\frac{\lambda_n}{e_n}\vert P^1_n-P^2_n\vert^2\rightarrow+\infty \ \ \ (\text{ if }k=2),
\end{equation}
\begin{equation}\label{blow5}
\frac{\lambda_n}{e_n}dist(P^i_n,\partial\Omega)^2\rightarrow+\infty,
\end{equation}
as $n\rightarrow+\infty$ and
\begin{equation}\label{blow6}
u_n(P_n^i)=\max_{\Omega\cap B_{R_n\lambda_n^{-1/2}\sqrt{e_n}}(P_n^i)}u_n,
\end{equation}
for some $R_n\rightarrow+\infty$ as $n\rightarrow+\infty$. Moreover, there holds
\begin{equation}\label{blow7}
u_n(x)\leq C\lambda_n^{\frac{1}{p-2}}\sum_{i=1}^{k}e^{-\gamma\lambda_n^{1/2}e_n^{-1/2}\vert x-P^i_n\vert}, \ \forall x\in\Omega, \ n\in\mathbb{N},
\end{equation}
for some $C,\gamma>0$.
\end{lemma}
\begin{proof}
The proof is divided into two steps.

\emph{Step 1  There exist $k\in\{1,2\}$, and sequences of points $\{P^i_n\}$, $i=\{1,k\}$ satisfying (\ref{blow4})-(\ref{blow6}) so that
\begin{equation}\label{blow8}
\lim_{R\rightarrow+\infty}\bigg(\limsup_{n\rightarrow+\infty}\bigg[\lambda_n^{-\frac{1}{p-2}}\max_{\{d_n(x)\geq R\varepsilon_n\sqrt{e_n}\}}u_n(x)\bigg]
\bigg)=0,
\end{equation}
where $d_n(x)=\min{\{\vert x-P_n^i\vert:i=1,k\}}$ is the distance function from $\{P_n^1,P_n^k\}$.}

Assume that $P_n^1$ is a point such that $u_n(P_n^1)=\max\limits_{\Omega}u_n$. If (\ref{blow8}) is satisfied for $P_n^1$, then we set $k=1$. Clearly $P_n^1$ satisfies (\ref{blow6}). By Lemma \ref{lem:blow1}, one can check that (\ref{blow5}) holds, which means the claim holds. Otherwise, if $P_n^1$ does not satisfy (\ref{blow8}), we suppose, for some $\delta>0$,
\[
\limsup_{R\rightarrow+\infty}\bigg(\limsup_{n\rightarrow+\infty}\bigg[\varepsilon_n^{\frac{2}{p-2}}\max_{\{\vert x-P^1_n\vert\geq R\varepsilon_n\sqrt{e_n}\}}u_n\bigg]\bigg)=4\delta>0.
\]
In view of Lemma \ref{lem:blow1}, we infer that up to subsequence,
\begin{equation}\label{blow9}
\varepsilon_n^{\frac{2}{p-2}}u_n(\varepsilon_n\sqrt{e_n}y+P_n^1):=U_n^1(y)\rightarrow U(y) \ \ \ \text{in }C^1_{\text{loc}}(\mathbb{R}^N)
\end{equation}
as $n\rightarrow+\infty$. Since $U\rightarrow0$ as $\vert x\vert\rightarrow+\infty$, there exists $R$ large such that
\begin{equation}\label{blow10}
U(y)\leq\delta, \ \ \ \forall\vert y\vert\geq R.
\end{equation}
Then, taking $R$ larger if necessary, we may suppose that, up to subsequence,
\begin{equation}\label{blow11}
\varepsilon_n^{\frac{2}{p-2}}\max_{\{\vert x-P^1_n\vert\geq R\varepsilon_n\sqrt{e_n}\}}u_n\geq2\delta.
\end{equation}
Considering $u_n=0$ on $\partial\Omega$, there exists $P_n^2\in \Omega\backslash B_{R\varepsilon_n\sqrt{e_n}}(P_n^1)$, such that
\[
u_n(P_n^2)=\max_{\Omega\backslash B_{R\varepsilon_n\sqrt{e_n}}(P_n^1)}u_n.
\]
Then we have
\begin{equation}\label{blow1-28}
\frac{\vert P_n^2-P_n^1\vert}{\varepsilon_n\sqrt{e_n}}\rightarrow+\infty.
\end{equation}
Indeed, suppose by contradiction that
\[
\frac{\vert P_n^2-P_n^1\vert}{\varepsilon_n\sqrt{e_n}}\rightarrow R'\geq R,
\]
using (\ref{blow9})-(\ref{blow10}) one has
\[
\varepsilon_n^{\frac{2}{p-2}}u_n(P_n^2)=U^1_n(\frac{P_n^2-P_n^1}{\varepsilon_n\sqrt{e_n}})\rightarrow U(R')\leq\delta,
\]
which contradicts to (\ref{blow11}). Hence (\ref{blow1-28}) is valid. As a consequence, we obtain (\ref{blow4}) for $\{P^1_n,P^2_n\}$.

 Putting
\[
\tilde{\varepsilon}_{n,2}=u_n(P_n^2)^{-\frac{p-2}{2}}  \ \ \ \text{and} \ \ \   R_{n,2}=\frac{1}{2}\frac{\vert P_n^2-P_n^1\vert}{\tilde{\varepsilon}_{n,2}\sqrt{e_n}},
\]
we conclude $\tilde{\varepsilon}_{n,2}\leq(2\delta)^{-\frac{p-2}{2}}\varepsilon_n$ by (\ref{blow11}), and thus
\[
R_{n,2}\geq\frac{(2\delta)^{\frac{p-2}{2}}}{2}\frac{\vert P_n^2-P_n^1\vert}{\varepsilon_n\sqrt{e_n}}\rightarrow+\infty,
\]
 which implies
\[
u_n(P_n^2)=\max_{\Omega\cap B_{R_{n,2}\tilde{\varepsilon}_{2,n}\sqrt{e_n}}(P_n^2)}u_n.
\]
In fact, using $\varepsilon_n\sqrt{e_n}<<\vert P_n^2-P_n^1\vert~$, it follows that
\[
\vert x-P_n^1\vert~\geq\vert P_n^2-P_n^1\vert-\vert x-P_n^2\vert\geq\frac{1}{2}\vert P_n^2-P_n^1\vert\geq R\varepsilon_n\sqrt{e_n}
\]
hold for all $x\in B_{R_{n,2}\tilde{\varepsilon}_{n,2}\sqrt{e_n}}(P_n^2)$.
Hence, $\Omega\cap B_{R_{n,2}\tilde{\varepsilon}_{n,2}\sqrt{e_n}}(P_n^2)\subset\Omega\backslash B_{R\varepsilon_n\sqrt{e_n}}(P_n^1)$. By Lemma \ref{lem:blow1} (i), we derive that (\ref{blow5}) and (\ref{blow6}) hold for $\{P_n^1, P_n^2\}$. Thus the conclusion is valid if (\ref{blow8}) holds. If not, we assume that the sequences $P_n^1,P^2_n,P_n^3$ satisfy (\ref{blow4})-(\ref{blow6}), but do not satisfy (\ref{blow8}). Similar to the above argument, there exists $R>0$ large enough such that up to subsequence,
\[
\varepsilon_n^{\frac{2}{p-2}}\max_{\{d_n(x)\geq R\varepsilon_n\sqrt{e_n}\}}u_n(x)\geq2\delta,
\]
where $d_n(x)=\min{\{\vert x-P_n^i\vert~:i=1,2\}}$. By Lemma \ref{lem:blow1} (iii), we deduce
\begin{equation}\label{blow12}
\varepsilon_n^{\frac{2}{p-2}}u_n(\varepsilon_n\sqrt{e_n}y+P_n^i):=U_n^i(y)\rightarrow U(y) \ \ \ \text{in }C^1_{\text{loc}}(\mathbb{R}^N)
\end{equation}
as $n\rightarrow+\infty$. Using $U\rightarrow0$ as $\vert x\vert~\rightarrow+\infty$ again, there exists $R$ large enough such that $U(y)\leq\delta$ for $\vert y\vert~\geq R$. Next, we reiterate the preceding arguments, replacing $\vert x-P_n^1\vert$ with $d_n(x)$. Let $P_n^{3}$ such that
\begin{equation}\label{blow13}
u_n(P_n^{3})=\max_{\{d_n(x)\geq R\varepsilon_n\sqrt{e_n}\}}u_n\geq2\delta\varepsilon_n^{-\frac{2}{p-2}}.
\end{equation}
Using (\ref{blow12}) it follows that
\[
\frac{\vert P_n^{3}-P_n^{i}\vert}{\varepsilon_n\sqrt{e_n}}\rightarrow+\infty \ \ \ \text{as }n\rightarrow+\infty
\]
for all $i=1,2$, so (\ref{blow4}) holds for $\{P_n^1,P_n^2,P_n^{3}\}$.

 Letting
\[\tilde{\varepsilon}_{n,3}=u_n(P_n^{3})^{-\frac{p-2}{2}}  \ \ \ \text{and} \ \ \  R_{n,3}=\frac{1}{2}\frac{d_n(P_n^{3})}{\tilde{\varepsilon}_{n,3}\sqrt{e_n}},
\]
 it follows from (\ref{blow13}) that
$\tilde{\varepsilon}_{n,3}\leq(2\delta)^{-\frac{p-2}{2}}\varepsilon_n$.
Then, we have $R_{n,3}\rightarrow+\infty$ as $n\rightarrow+\infty$. As before, taking into account
\[
u_n(P_n^{3})=\max_{\Omega\cap B_{R_{n,3}\tilde{\varepsilon}_{n,3}\sqrt{e_n}}(P_n^{3})}u_n,
\]
by Lemma \ref{lem:blow1}, we deduce that (\ref{blow5}) and (\ref{blow6}) are valid for $\{P_n^1,P_n^2,P_n^3\}$.

For $P_n^i,i=1,2,3$, using Lemma \ref{lem:blow1} again, we can find $\phi^i_n\in C^\infty_0(\Omega)$ with supp$\phi^i_n\subset B_{R\varepsilon_n\sqrt{e_n}}(P_n^i)$, $R>0$, such that
\[
(a+be_n)\int_\Omega\vert \nabla\phi^i_n\vert~^2dx+\int_\Omega[(\lambda-(p-1)\rho_nu_n^{p-2})(\phi^i_n)^2]dx<0.
\]
From (\ref{blow4}), we deduce $\phi^1_n,\phi^2_n,\phi_n^{3}$ are
mutually orthogonal when $n$ is large enough, which implies
\[
\lim_{n\rightarrow+\infty}m(u_n)\ge3,
\]
this gives a contradiction with $m(u_n)\leq2$. Therefore, the argument must end after at most two iterations.
\medskip

\emph{Step 2  There exist $\gamma, C>0$ such that
\[
u_n(x)\leq C\lambda_n^{\frac{1}{p-2}}\sum_{i=1}^{k}e^{-\gamma\lambda_n^{1/2}e_n^{-1/2}\vert x-P_n^i\vert~}, \ \ \ \forall x\in\Omega,n\in\mathbb{N}.
\]
}
From (\ref{blow8}), there exists $R>0$ large enough such that for $n\geq n(R)$
\[
\lambda_n^{-\frac{1}{p-2}}\max_{\{d_n(x)\geq R\lambda_n^{-1/2}\sqrt{e_n}\}}u_n(x)\leq\bigg(\frac{1}{4p}\bigg)^{\frac{1}{p-2}},
\]
where $d_n(x)=\min{\{\vert x-P_n^i\vert: i=1,k\}}$. Let $$A_n:=\{d_n(x)\geq R\lambda_n^{-1/2}\sqrt{e_n}\}$$ for $n\geq n(R)$. Then we get
\begin{equation}\label{blow14}
\frac{\lambda_n}{2}-p\rho_nu_n^{p-2}(x)\geq\frac{\lambda_n}{4}, \forall x\in A_n.
\end{equation}
Setting \
\[
\phi^i_n(x):=e^{-\gamma\lambda_n^{1/2}e_n^{-1/2}\vert x-P_n^i\vert},
\]
where $0<\gamma\leq\frac{1}{2\sqrt{1+b}}$,
by direct calculus, one has
\begin{eqnarray*}
&&[-(a+be_n)\Delta+\frac{\lambda_n}{2}-p\rho_nu_n^{p-2}(x)](\phi^i_n)\notag\\
&=&\lambda_n\phi^i_n\bigg[(ae_n^{-1}+b)\bigg(-\gamma^2+\frac{(N-1)\gamma}{\lambda_n^{1/2}e_n^{-1/2}\vert x-P_n^i\vert}\bigg)+\frac{1}{2}-p\rho_n\lambda_n^{-1}u_n^{p-2}(x)\bigg]\notag\\
&\geq&0\notag
\end{eqnarray*}
in $A_n$ for $n$ large. By $U(x)\to0$ as $\vert x\vert\to+\infty$, it follows that, for $R>0$ large enough,
\[
\bigg(e^{\gamma R}\phi^i_n(x)-\lambda_n^{-\frac{1}{p-2}}u_n(x)\bigg)\bigg\vert~_{\partial B_{R\lambda_n^{-1/2}\sqrt{e_n}}(P_n^i)}\rightarrow1-U(R)>0
\]
as $n\rightarrow+\infty$.

Next, we define
\[
\phi_n:=e^{\gamma R}\lambda_n^{\frac{1}{p-2}}\sum_{i=1}^{k}\phi^i_n.
\]
Considering $L_n=-(a+be_n)\Delta+\lambda_n-\rho_nu_n^{p-2}$, one can easily see that
\[
L_n(\phi_n-u_n)\geq0 \ \ \ \text{in } \{d_n(x)>R\lambda_n^{-1/2}\sqrt{e_n}\}
\]
and $\phi_n-u_n\geq0$ in $\partial A_n\cup\partial\Omega$. In view of (\ref{blow4})-(\ref{blow6}), we obtain
\[
\partial A_n=\bigcup_{i=1}^{k}\partial B_{R\lambda_n^{-1/2}\sqrt{e_n}}(P_n^i)\subset\Omega.
\]
 Then, it follows from the maximum principle that, for any $x\in A_n$,
\[
u_n(x)\leq\phi_n(x)=e^{\gamma R}\lambda_n^{\frac{1}{p-2}}\sum_{i=1}^{k}e^{-\gamma\lambda_n^{1/2}e_n^{-1/2}\vert x-P_n^i\vert~}
\]
By $\lambda_n\tilde{\varepsilon}_n\rightarrow\tilde{\lambda}$, there exists $C>0$, such that
\[
u_n(x)\leq\max_{\Omega}u_n=u_n(P_n)=(\tilde{\varepsilon}_n)^{-\frac{2}{p-2}}\leq Ce^{\gamma R}\lambda_n^{\frac{1}{p-2}}\sum_{i=1}^{k}e^{-\gamma\lambda_n^{1/2}e_n^{-1/2}\vert x-P_n^i\vert~}
\]
if $x\in A_n^c$ when $n$ is large enough. Thus, taking a subsequence if necessary, for some $C>0$, (\ref{blow7}) holds for any $n\in\mathbb{N}$ and the proof is complete.
\end{proof}

\section{Proof of Theorem 1}\label{sec:5}

In this section, we complete the proof of Theorem 1. Let $0<c\leq c^*$. In view of Lemma \ref{lem:mimax g}, using the monotonicity of $c_{\rho_n}$, we get
\[
J_1(w_1)\leq J_{\rho}(w_1)\leq c_{\rho_n}\leq c_{\frac{1}{2}},
\]
which implies the boundedness of $c_{\rho_n}$. Then, we have the following result.

\begin{lemma}\label{lem:boundedness}
Let $\{u_n\}\subset H^1_0(\Omega)$ be a sequence of solutions to (\ref{eq5.1}) for some $\{\lambda_n\}\subset\mathbb{R}$ and $\rho_n\rightarrow1^-$. Assume that
\[
\int_{\Omega}\vert u_n\vert^2dx=c \ \ \ \text{and} \ \ \ m(u_n)\leq2, \ \ \ \forall n,
\]
and assume the energy levels $\{c_n=J_{\rho_n}(u_n)\}$ is bounded. Then, the sequences $\{\lambda_n\}\subset\mathbb{R}$ and $\{u_n\}\subset H^1_0(\Omega)$ must be bounded. In addition, $\{u_n\}\subset \mathcal{S}_c$ is a bounded Palais-Smale sequence for $J$.
\end{lemma}
\begin{proof}
By (\ref{energy}) we derive that if $\{\lambda_n\}$ is bounded, then $\{u_n\}$ is bounded in $H^1_0(\Omega)$ and thus a Palais-Smale sequence for $J$ constrained on $\mathcal{S}_c$.

It suffices to show the boundedness of $\{\lambda_n\}$. By contradiction, we assume that, up to a subsequence, $\lambda_n\rightarrow+\infty$. Let $P_n^i$ denotes the local maxium point of the blow up function $U_n^i$ as defined in Lemma \ref{lem:blow1}, and $P^i$ denotes the corresponding limit point, where $i=1,k$ with $k\in\{1,2\}$.

We claim that, for any $R>0$,
\begin{equation}\label{opro2}
\bigg\vert\lambda_n^{\frac{N}{2}-\frac{2}{p-2}}e_n^{-\frac{N}{2}}\int_{\Omega}u_n^2dx-\sum_{i=1}^{k}\int_{B_R({P}_n^i)}(U_n^i)^2dx\bigg\vert\rightarrow+\infty.
\end{equation}
Indeed, by Lemma \ref{lem:rel2} it follows that
\begin{equation}\label{opro-1-28}
\lambda_n^{\frac{N}{2}-\frac{2}{p-2}}e_n^{-\frac{N}{2}}\int_{\Omega}u_n^2dx=\lambda_n^{\frac{N}{2}-\frac{2}{p-2}}e_n^{-\frac{N}{2}}c\rightarrow+\infty.
\end{equation}
Using Lemma \ref{lem:blow1}, we deduce
\[
\sum_{i=1}^{k}\int_{B_R({P}_n^i)}(U_n^i)^2dx\rightarrow\int_{B_R({P}^i)}(U^i)^2dx,
\]
which together with (\ref{opro-1-28}) shows that (\ref{opro2}) holds.

On the other hand, from Lemma \ref{lem:blow2} we infer
\begin{eqnarray*}
&&\bigg\vert\lambda_n^{\frac{N}{2}-\frac{2}{p-2}}e_n^{-\frac{N}{2}}\int_{\Omega}u_n^2dx-\sum_{i=1}^{k}\int_{B_{R}({P}_n^i)}(U_n^i)^2dx\bigg\vert\\
&=&\lambda_n^{\frac{N}{2}-\frac{2}{p-2}}e_n^{-\frac{N}{2}}\bigg\vert\int_{\Omega}u_n^2dx-\sum_{i=1}^{k}\int_{B_{R\lambda_n^{-1/2}e_n^{1/2}}({P}_n^i)}u_n^2dx\bigg\vert~\notag\\
&=&\lambda_n^{\frac{N}{2}-\frac{2}{p-2}}e_n^{-\frac{N}{2}}\int_{\Omega\backslash \bigcup\limits_iB_{R\lambda_n^{-1/2}e_n^{1/2}}({P}_n^i)}u_n^2dx\notag\\
&\leq& C\lambda_n^{\frac{N}{2}}e_n^{-\frac{N}{2}}\sum_{i=1}^k\int_{\Omega\backslash B_{R\lambda_n^{-1/2}e_n^{1/2}}({P}_n^i)}e^{-\gamma\lambda_n^{1/2}e_n^{-1/2}\vert x-P^i_n\vert}dx\notag\\
&\leq& C\lambda_n^{\frac{N}{2}}e_n^{-\frac{N}{2}}\int_{R\lambda_n^{-1/2}e_n^{1/2}}^{+\infty}e^{-2\gamma\lambda_n^{1/2}e_n^{-1/2}y}dy\notag\\
&\leq& C\int_{R}^{+\infty}e^{-2\gamma z}dz\notag\\
&\leq& Ce^{-2\gamma R}.
\end{eqnarray*}
This implies that
\[
\limsup_{n\rightarrow+\infty}\bigg\vert \lambda_n^{\frac{N}{2}-\frac{2}{p-2}}e_n^{-\frac{N}{2}}\int_{\Omega}u_n^2dx-\sum_{i=1}^{k}\int_{B_R({P}_n^i)}(U_n^i)^2dx\bigg\vert~\leq Ce^{-2\gamma R},
\]
which provides a contradiction with (\ref{opro2}).
\end{proof}

\begin{proof}[Proof of Theorem \ref{thm:1}]
By Lemma \ref{lem:morse}, for any $c\in (0, c^*]$, there exists a sequence $\rho_n\rightarrow1^-$ and a corresponding sequence $\{u_{n}\}\subset \mathcal{S}_c$, mountain pass-type critical points of $J_{\rho_n}$, such that $J_{\rho_n}(u_n)=c_{\rho_n}$ and $m(u_{n})\leq2$. Applying Lemma \ref{lem:boundedness}, it follows that $\{u_{n}\}$ is a bounded Palais-Smale sequence for $J$. Then in view of Lemma \ref{lem:conv}, we obtain the conclusion.
\end{proof}

\section{Asymptotic behavior of $u_b$ as $b\to0$}\label{sec:6}

In this section, we investigate the asymptotic behavior of normalized solution $u_b$ for (\ref{thm:1}) as $b\rightarrow0$. Without loss of generality, let's assume that $b_n\in(0,1]$ with $b_n\rightarrow0$
in the sequel. Moreover, we assume that $(u_{b_n},\lambda_{b_n})\in \mathcal{S}^+_c\times\mathbb{R}$ solves
\begin{equation}\label{eq:asymptotic}
\begin{cases}
-\big(a+b_n\int_{\Omega}\vert \nabla u\vert^2dx\big)\Delta u+\lambda u=\vert u\vert^{p-2}u& \text{ in }\Omega,\\
u=0 & \text{ on }\partial\Omega
\end{cases}
\end{equation}
and $c_{b_n}=J(u_{b_n})$, where $\mathcal{S}_c^+:=\{u\in \mathcal{S}_c\vert~u>0\}$. For notational convenience, we set $(u_n,\lambda_n)$ and $c_n$ instead of $(u_{b_n},\lambda_{b_n})$ and $c_{b_n}$.

 At first, we investigate the monotonicity of $c_b$ as $b\in (0,1]$. For convenience of statement, we denote the energy functional $J$ by $J_{\rho, b}$.
\begin{lemma}\label{lem:monotonicity}
$c_b$ is nondecreasing with respect to $b\in (0,1]$.
\end{lemma}
\begin{proof}
Let $0<b_1\leq b_2\leq1$. From the proof of Lemma \ref{lem:mimax g}, we observe that the path set $\Gamma$ is independent of $b$. Then, for any $\gamma\in\Gamma$, we can obtain
\[
c_{b_1}:=\inf_{\gamma\in\Gamma}\max_{t\in[0,1]}J_{1,b_1}(\gamma(t))\leq \inf_{\gamma\in\Gamma}\max_{t\in[0,1]}J_{1,b_2}(\gamma(t)):=c_{b_2}.
\]
\end{proof}

By the above lemma, we obtain that the energy level $c_n$ is bounded since
\[
J_{1,0}(w_1)\leq J_{1,b_n}(w_1)\leq c_{n}\leq c_{1},
\]
where $w_1$ is given in Lemma \ref{lem:mimax g}.
By the identity (\ref{energy}) and $p>4$, we know that $\lambda_n$ is bounded from below, and furthermore we have $\{u_n\}$ is bounded, provided that $\{\lambda_n\}$ is bounded.

Assume by contradiction that $\lambda_n\to+\infty$. Using (\ref{energy}) again it follows that $e_n\to+\infty$. In what follows, we will distinguish two cases:
Case (i). $b_ne_n\rightarrow e\geq0$; Case (ii). $b_ne_n\to+\infty$.

For Case (i), we shall perform a blow up argument as in Section \ref{sec:4} to produce a contradiction.

Firstly, we study the behavior of the sequence $\{u_n\}$ near the local maximum points.

\begin{lemma}\label{lem:reblow1}
Assume that $b_ne_n\rightarrow e\geq0$ as $n\rightarrow+\infty$. Assume $Q_n\in\Omega$ be such that, for some $R_n\rightarrow+\infty$,
\[
u_n(Q_n)=\max_{\Omega\cap B_{R_n\tilde{\tau}_n}(Q_n)}u_n \ \ \ \text{ where }\tilde{\tau}_n=(u_n(Q_n))^{-\frac{p-2}{2}}\rightarrow0.
\]
Set
\[
v_n(y)=\tau_n^{\frac{2}{p-2}}u_n(\tau_ny+Q_n)\text{ for }y\in\Omega_n=\frac{\Omega-Q_n}{\tau_n}\text{ with } \tau_n=\lambda_n^{-\frac{1}{2}}.
\]
Then, up to a subsequence, we have
\begin{itemize}
\item[(i)] $\frac{\tau_n}{d(Q_n,\partial\Omega)}\rightarrow0$ as $n\rightarrow+\infty$.
\item[(ii)] $u_n(Q_n)=\max\limits_{\Omega\cap B_{R_n\tau_n}(Q_n)}u_n$ for some $R_n\rightarrow+\infty$.
\item[(iii)] $v_n\rightarrow V$ in $C^1_{\text{loc}}(\mathbb{R}^N)$ as $n\rightarrow+\infty$, where $V$ solves
\begin{equation*}
\begin{cases}
-(a+e)\Delta V+V=V^{p-1}  & \text{ \rm{in} }\mathbb{R}^N,\\
0<V\leq V(0)  & \text{ \rm{in} }\mathbb{R}^N, \\
V\rightarrow0  & \text{ \rm{as} }\vert x\vert\rightarrow+\infty.
\end{cases}
\end{equation*}
\item[(iv)] There exists $\phi_n\in C^\infty_0(\Omega)$ with supp$\phi_n\subset B_{R\tau_n}(Q_n)$, $R>0$, such that for $n$ large
\begin{equation*}
(a+b_ne_n)\int_\Omega\vert \nabla\phi_n\vert^2dx+\int_\Omega[(\lambda-(p-1)u_n^{p-2})\phi_n^2]dx<0.
\end{equation*}
\item[(v)] For all $R>0$ and $q\geq1$ there holds
\begin{equation*}
\lim_{n\rightarrow+\infty}\lambda_n^{\frac{N}{2}-\frac{q}{p-2}}\int_{B_{R\tau_n}(Q_n)}u_n^qdx=\int_{B_{R}(0)}V^qdx.
\end{equation*}
\end{itemize}
\end{lemma}
\begin{proof}
Let $\tilde{v}_n$ be defined by
\[
\tilde{v}_n(y)=\tilde{\tau}_n^{\frac{2}{p-2}}u_n(\tilde{\tau}_ny+Q_n) \ \ \ \text{for } y\in\tilde{\Omega}_n:=\frac{\Omega-Q_n}{\tilde{\tau}_n}.
\]
Notice that the function $\tilde{v}_n$ satisfies
\begin{equation*}
\begin{cases}
-(a+b_ne_n)\Delta\tilde{v}_n+\lambda_n\tilde{\tau}_n^2\tilde{v}_n=\tilde{v}_n^{p-1}  & \text{ in }\tilde{\Omega}_n,\\
0<\tilde{v}_n\leq \tilde{v}_n(0)=1  & \text{ in }\tilde{\Omega}_n\cap B_{R_n}(0), \\
\tilde{v}_n=0  & \text{ on }\partial\tilde{\Omega}_n.
\end{cases}
\end{equation*}
Since $Q_n$ is the local maximum point of $v_n$, we obtain
\[
0\leq-(a+b_ne_n)\Delta\tilde{v}_n(0)=1-\lambda_n\tilde{\tau}_n^2,
\]
which gives $\lambda_n\tilde{\tau}_n^2\in[0,1]$. Then, there exists $\hat\lambda\in[0,1]$ such that up to subsequence, $\lambda_n\tilde{\tau}_n\to\hat\lambda$ as $n\to+\infty$. Now by elliptic regularity, we can derive that up to subsequence, $\tilde{v}_n\to\tilde{v}$ in $C^1_{\text{loc}}(\bar{H})$ as $n\rightarrow+\infty$, where $\tilde{v}$ satisfies
\begin{equation*}
\begin{cases}
-(a+e)\Delta\tilde{v}+\hat\lambda\tilde{v}=\tilde{v}^{p-1}  & \text{ in }H,\\
0<\tilde{v}\leq \tilde{v}(0)=1  & \text{ in }H, \\
\tilde{v}=0  & \text{ on }\partial H.
\end{cases}
\end{equation*}
Here $H$ is a half-space or $\mathbb{R}^N$. By employing similar arguments as those in Lemma \ref{lem:blow1}, we deduce that $H=\mathbb{R}^N$ and $\hat\lambda>0$. Furthermore, we point out that the remaining part of the proof can be finalized  by drawing parallels to the proof of Lemma \ref{lem:blow1}.
\end{proof}

Furthermore, we give the global behavior of the sequence $\{u_n\}$. Since the proof is analogous to Lemma \ref{lem:blow2}, so we only propose lemma here and omit the proof.

\begin{lemma}\label{lem:reblow2}
Assume that $b_ne_n\rightarrow e\geq0$ as $n\rightarrow+\infty$. Then, there exists $k\in\{1,2\}$, and sequences of points $\{Q^i_n\}$, $i=\{1,k\}$, such that
\begin{equation*}
\lambda_n \vert Q^1_n-Q^2_n\vert^2\rightarrow+\infty \ \ \ (\text{ if }k=2),
\end{equation*}
\begin{equation*}
\lambda_ndist(Q^i_n,\partial\Omega)^2\rightarrow+\infty,
\end{equation*}
as $n\rightarrow+\infty$ and
\begin{equation*}
u_n(Q_n^i)=\max_{\Omega\cap B_{R_n\lambda_n^{-1/2}}(Q_n^i)}u_n,
\end{equation*}
for some $R_n\rightarrow+\infty$ as $n\rightarrow+\infty$. Moreover, there holds
\begin{equation*}
u_n(x)\leq C\lambda_n^{\frac{1}{p-2}}\sum_{i=1}^{k}e^{-\gamma\lambda_n^{1/2}\vert x-Q^i_n\vert~}, \ \forall x\in\Omega, \ n\in\mathbb{N},
\end{equation*}
for some $C,\gamma>0$.
\end{lemma}

At this point, we can exploit the proof of Theorem \ref{thm:1} with minor changes to find the contradiction (here we make use of Lemma \ref{lem:reblow1} and Lemma \ref{lem:reblow2} instead of Lemma \ref{lem:blow1} and Lemma \ref{lem:blow2}). Therefore, $\{\lambda_n\}$ is bounded in the case $b_ne_n\to e\geq0$.

Next, let's investigate Case (ii), i.e., $b_ne_n\to+\infty$. Notice that the function $\tilde{V}_n$ defined by
\[
\tilde{V}_n(y):=\tilde{\tau}_n^\frac{2}{p-2}u_n(\tilde{\tau}_n\sqrt{e_n}y+Q_n) \ \ \ \text{for } y\in\tilde{\Omega}_n:=\frac{\Omega-Q_n}{\tilde{\tau}_n\sqrt{e_n}},
\]
satisfies
\begin{equation*}
\begin{cases}
-(ae^{-1}_n+b_n)\Delta\tilde{V}_n+\lambda_n\tilde{\tau}_n^2\tilde{V}_n=\tilde{V}_n^{p-1}  & \text{ in }\tilde{\Omega}_n,\\
0<\tilde{V}_n\leq \tilde{V}_n(0)=1  & \text{ in }\tilde{\Omega}_n\cap B_{R_n}(0), \\
\tilde{V}_n=0  & \text{ on }\partial\tilde{\Omega}_n.
\end{cases}
\end{equation*}
Similar to the proof of Lemma \ref{lem:blow1}, we can derive $\lambda_n\tilde{\tau}_n^2\to \check\lambda\in[0,1]$ and $\tilde{V}_n\to\tilde{V}$ in $C^1_{\text{loc}}(\bar{H})$ as $n\rightarrow+\infty$, where $\tilde{V}$ satisfies
\begin{equation*}
\begin{cases}
\check \lambda\tilde{V}=\tilde{V}^{p-1}  & \text{ in }H,\\
0<\tilde{V}\leq \tilde{V}(0)=1  & \text{ in }H, \\
\tilde{V}=0  & \text{ on }\partial H.
\end{cases}
\end{equation*}
Here $H$ is a half-space or $H=\mathbb{R}^N$, which is impossible. As a consequence, $\{\lambda_n\}$ is bounded as $b_ne_n\to+\infty$.

\begin{proof}[Completion of proof of Theorem \ref{thm:2}]
Based on the aforementioned analysis, we can obtain that $\{\lambda_n\}$ is bounded. Then, there exists $\lambda_0\in\mathbb{R}$ such that up to a subsequence, $\lambda_n\rightarrow \lambda_0$. From (\ref{energy}), we derive that $\{u_n\}$ is bounded. By Lemma \ref{lem:conv}, there exists $u_0\in \mathcal{S}^+_c$ such that
\begin{eqnarray*}
&& u_n\rightarrow u_0~~\mbox{in}~~ H_0^{1}(\Omega),\\
&& u_n\rightarrow u_0~~\mbox{in}~~L^{p}(\Omega),~~p\in[2,2^\ast),\\
&& u_n(x)\rightarrow u_0(x)~~\mbox{for a.e. }x\in\Omega.
\end{eqnarray*}
Since $\{u_n,\lambda_n\}$ solves (\ref{eq:asymptotic}), for any $v\in H^1_0(\Omega)$, we have
\[
a\int_{\Omega}\nabla u_n\nabla vdx+b_n\int_{\Omega}\vert \nabla u_n\vert^2dx\int_{\Omega}\nabla u_n\nabla vdx+\lambda_n\int_{\Omega}u_nvdx=\int_{\Omega}\vert u_n\vert^{p-2}u_nvdx.
\]
Thus, one can see that
\[
a\int_{\Omega}\nabla u_0\nabla vdx+\lambda_0\int_{\Omega}u_0vdx=\int_{\Omega}\vert u_0\vert^{p-2}u_0vdx,
\]
which implies that $u_0\subset \mathcal{S}_c^+$ is a positive normalized solution of (\ref{e2}).
This completes the proof.
\end{proof}

\section*{Acknowledgments}\setcounter{equation}{0}
This work is supported by NSFC Grant (No. 12471102, 11471067) and the Beijing Natural Science Foundation (No. 1242007).
\\

\end{document}